\documentclass[a4paper,10pt]{amsart}
\usepackage[english]{babel}
\usepackage{amsmath,amsthm,amssymb,amsfonts}
\usepackage{multicol}
\usepackage[titletoc]{appendix}

\usepackage[pdftex]{color}
\usepackage[bookmarks=true,hyperindex,pdftex,colorlinks,citecolor=blue,
linkcolor=blue,urlcolor=blue]{hyperref}
\usepackage[dvipsnames]{xcolor}
\usepackage{mathtools}
\usepackage{graphicx}
\usepackage{caption}
\usepackage{subcaption}
\usepackage{soul}
\usepackage[normalem]{ulem}
\usepackage{enumitem}

\usepackage{tikz}
 
\usetikzlibrary{arrows,automata}
\usetikzlibrary{positioning}

\tikzset{
	confetti/.style={
		rectangle,
		draw=black, very thick,
		minimum height=4em,
		minimum width=6em,
		inner sep=1pt,
		text centered,
	},
	fletxa/.style={
		draw=black, very thick, ->
	},
}

\theoremstyle{definition}
\newtheorem{theorem}{Theorem}[section]

\newtheorem{prop}[theorem]{Proposition}
\newtheorem{lemma}[theorem]{Lemma}

\newtheorem{definition}[theorem]{Definition}

\newcommand{\R}{\mathbb{R}}

\newcommand{\C}{\mathbb{C}}
\newcommand{\K}{\mathbb{K}}

\newcommand{\e}{\varepsilon}
\newcommand{\eps}{\varepsilon}

\newcommand{\ran}{\text{ran}}

\DeclareMathOperator{\re}{Re}

\DeclareMathOperator{\Id}{Id}

\renewcommand{\leq}{\leqslant}
\renewcommand{\geq}{\geqslant}

\renewcommand{\geq}{\geqslant}

\begin{document}

	\title{The Bishop-Phelps-Bollob\'as properties in complex Hilbert spaces}

	\author[Choi]{Yun Sung Choi}
\address[Choi]{Department of Mathematics, POSTECH, Pohang 790-784, Republic of Korea}
\email{\texttt{mathchoi@postech.ac.kr}}

	\author[Dantas]{Sheldon Dantas}
	\address[Dantas]{Department of Mathematics, Faculty of Electrical Engineering, Czech Technical University in Prague, Technick\'a 2, 166 27 Prague 6, Czech Republic  \newline
\href{http://orcid.org/0000-0001-8117-3760}{ORCID: \texttt{0000-0001-8117-3760}  } }
\email{\texttt{gildashe@fel.cvut.cz}}

	\author[Jung]{Mingu Jung}
	\address[Jung]{Department of Mathematics, POSTECH, Pohang 790-784, Republic of Korea \newline
	\href{http://orcid.org/0000-0000-0000-0000}{ORCID: \texttt{0000-0003-2240-2855}  }}
\email{\texttt{jmingoo@postech.ac.kr}}

	\begin{abstract} In this paper, we consider the {\it Bishop-Phelps-Bollob\'as point property} for various classes of operators on complex Hilbert spaces, which is a stronger property than the Bishop-Phelps-Bollob\'{a}s property. We also deal with analogous problem by replacing the norm of an operator with its numerical radius.
	\end{abstract}

	\date{\today}
	
\thanks{The first author was supported by Basic Science Research Program through the National Research Foundation of Korea (NRF) funded by the Ministry of Education (NRF-2015R1D1A1A09059788 and NRF-2018R1A4A1023590). The second author was supported by the project OPVVV CAAS CZ.02.1.01/0.0/0.0/16\_019/0000778, Centrum pokro\v{c}il\'ych aplikovan\'ych p\v{r}\'irodn\'ich v\v{e}d (Center for Advanced Applied Science), by Pohang Mathematics Institute (PMI), POSTECH, Korea and by NRF funded by the Ministry of Education, Science and Technology (NRF-2015R1D1A1A09059788). The third author was supported by NRF (NRF-2015R1D1A1A09059788).}
	

	\maketitle
	
	\thispagestyle{plain}
	

	\section{Introduction} \label{int}
	
The study of the denseness of norm-attaining operators between Banach spaces was motivated by the celebrated Bishop-Phelps theorem \cite{BP} published in 1961. J. Lindenstrauss \cite{Lin} showed in 1963 not only that such a denseness does not hold in general, but also that if a Banach space $X$ is reflexive, then it holds for operators from $X$ into an arbitrary Banach space $Y$. After that, this result was improved by J. Bourgain \cite{Bou}. He showed that if $X$ is a Banach space with the Radon-Nikod\'{y}m property, then every bounded (compact) operator $T$ from $X$ into an arbitrary Banach space $Y$ can be approximated by norm-attaining (compact) operators $T+K$ with a finite rank operator $K$. A few years later, C. Stegall observed that the above $K$ can be chosen to be a rank one operator \cite{Ste}. 
There is a vast literature about this topic and we suggest the reader the survey paper \cite{Aco}. 

On the other hand, B. Bollob\'{a}s \cite{Bol} refined in 1970 the Bishop-Phelps theorem quantitatively by showing that both functionals and points where they almost attain the norm can be approximated by norm-attaining functionals and points where they do attain the norm. In 2008, M. Acosta, R. Aron, D. Garc\'ia, and M. Maestre began studying this theorem for operators between Banach spaces $X$ and $Y$, and introduced the Bishop-Phelps-Bollob\'as property (see \cite[Definition 1.1]{AAGM}): we say that the pair $(X, Y)$ has the {\it Bishop-Phelps-Bollob\'as property} (BPBp, for short) if given $\e > 0$, there is $\eta(\e) > 0$ such that whenever $T \in \mathcal{L}(X, Y)$ with $\|T\| =1$ and $x_0 \in S_X$ satisfy
$\|T x_0\| > 1 - \eta(\e),$
there are $S \in \mathcal{L}(X, Y)$ with $\|S\| = 1$ and $x_1 \in S_X$ such that
\begin{equation*} 
\|S x_1\| = 1, ~~\|x_1 - x_0\| < \e, ~~ \mbox{and} ~~\|S - T\| < \e.
\end{equation*} 
Here, $\mathcal{L}(X, Y)$ denotes the Banach space of all bounded linear operators from $X$ into $Y$ and $S_X$ the unit sphere of $X$. When $X=Y$, $\mathcal{L}(X,Y)$ is abbreviated to $\mathcal{L}(X)$ and we simply say that $X$ has the BPBp when the pair $(X, X)$ has the BPBp.
With this definition, the refinement given by B. Bollob\'{a}s \cite{Bol} means simply that the pair $(X, \K)$ has the BPBp for every Banach space $X$, where $\K$ is either $\R$ or $\C$. Although there has been an extensive research on this property (see, for example, \cite{AcoFakSole, AMS, DGMM, KL}), we would like to focus on the case when $X$ is a complex Hilbert space $H$ by considering classical operators on $H$.

It was showed in 2012 by L. Cheng and Y. Dong \cite{CD} that the complex Hilbert space $H$ satisfies the BPBp for normal operators, that is, given $0< \e < 1/2,$ a normal operator $T\in \mathcal{L}(H)$ with $\|T\| = 1$ and $x_0 \in S_H$ such that $\|Tx_0\|> 1-\e$, there exist a normal operator $S\in \mathcal{L}(H)$ with $\|S\| = 1$ and $x_1\in S_H$ such that $\|S x_1\|=1$, $\|x_1 - x_0\| \leq \sqrt{2\e} + \sqrt[4]{2\e}$, and $\|S-T\| < \sqrt{2\e}$. To consider $H$ as a complex space is essential for that proof since spectral theory is used strongly. The analogous result for self-adjoint operators was obtained in 2014 by D. Garc\'{i}a, H.J. Lee, and M. Maestre \cite{GLM}. They also proved that $H$  has the  BPBp for Schatten-von Neumann operators even with respect to the Schatten $p$-norm $\sigma_p( \cdot )$. Moreover, $H$ satisfies the BPBp for compact operators as a particular case of a more general result: if $X$ is uniformly convex, then the pair $(X,Y)$ has the BPBp for compact operators for every $Y$ \cite{DGMM}.

In this paper, we study a stronger property, so-called the {\it Bishop-Phelps-Bollob\'as point property} for operators defined on complex Hilbert spaces such as positive, self-adjoint, anti-symmetric, unitary, compact, normal, and Schatten-von Neumann operators as well as some intersections between some of these classes.  We say that the pair $(X, Y)$ satisfies the {\it Bishop-Phelps-Bollob\'as point property} (BPBpp, for short) if given $\e > 0$, there is $\eta(\e) > 0$ such that whenever $T \in \mathcal{L}(X, Y)$ with $\|T\| = 1$ and $x_0 \in S_X$ satisfy $\|T x_0 \| > 1 - \eta(\e),$
there is $S \in \mathcal{L}(X, Y)$ with $\|S\| = 1$ such that 
$\|S x_0\| = 1 $ and $\|S - T\| < \e.$ 
This property was introduced in \cite{DKL} (see also \cite{DKKLM} for more recent results).

In parallel with the study of denseness of norm-attaining operators, a lot of attention was given also to the study of the denseness of numerical radius attaining operators. O. Toeplitz \cite{T} defined in 1918 the numerical range for matrices  which could be naturally extended for bounded operators on the Hilbert space $H$. The numerical range of $T$ is defined by $W(T) = \{ \langle Tx, x \rangle: x \in S_H \}$ and its numerical radius by $\nu(T) = \sup \{ |\lambda|: \lambda \in W(T)\} = \sup \{ | \langle Tx, x \rangle|: x \in S_H \}$, where the symbol $\langle \ , \ \rangle$ stands for the inner product on $H$.
Note that $\nu$ is a seminorm on $\mathcal{L}(H)$ satisfying $\nu(T) \leq \|T\|$ for every $T \in \mathcal{L}(H)$. It is well-known that for a complex Hilbert space $H$ with dimension greater than $1$, we always have $\|T\| \leq 2 \nu(T)$  for every $T \in \mathcal{L}(H)$ (see \cite{H}, pg. 114), which, on the other hand, it is not true for real Hilbert spaces. Recall that an operator on $H$ attains the numerical radius if there is $x_0 \in S_H$ such that $|\langle Tx_0, x_0 \rangle| = \nu(T)$. These concepts can be extended for a general Banach space (see \cite{B, L}). For instance, the numerical radius of an operator $T \in \mathcal{L}(X)$ is defined by $\nu (T) = \sup\{ |x^* (Tx)| : x \in S_X, x^* \in S_{X^*}, x^* (x) =1\}$. We refer the reader to the classical books \cite{BD1, BD2} for a complete background on the numerical range theory.

B. Sims showed that every self-adjoint operator on a Hilbert space can be approximated by self-adjoint operators each of which attains the numerical radius \cite[Theorem 3.9]{Sims} and I. Berg and B. Sims proved the denseness of numerical radius attaining operators on a uniformly convex space \cite{BS}. Also, many Banach spaces, such as $c_0, \ell_1, C(K)$ (where $K$ is a compact Hausdorff space), $L_1 (\mu)$, uniformly smooth Banach spaces, and Banach spaces with the Radon-Nikod\'{y}m property were shown to satisfy the property that the set of the numerical radius attaining operators is dense in the space of all bounded linear operators (see \cite{A, AP, C1, C2, C3}). 

Motivated by the BPBp, some authors studied the Bishop-Phelps-Bollob\'as property for numerical radius (see, for instance, \cite{F, GK,KLM1}) by considering the numerical radius of an operator instead of its norm. We say that a Banach space $X$ has the {\it Bishop-Phelps-Bollob\'as property for numerical radius} (the BPBp-$\nu$, for short) if given $\e > 0$, then there is $\eta(\e) > 0$ such that whenever $T \in \mathcal{L}(X)$ with $\nu(T) = 1$ and $(x, x^*) \in S_X \times S_{X^*}$ with $x^*(x) = 1$ satisfy
\begin{equation*}
|x^*(T x)| >  1 - \eta(\e),
\end{equation*}	
there exist $S \in \mathcal{L}(X)$ with $\nu(S) = 1$ and $(z, z^*) \in S_X \times S_{X^*}$ with $z^*(z) = 1$ such that
\begin{equation*}
|z^*(S z )| = 1, \ \ \  \|z^* - x^*\| < \e, \ \ \ \|z - x\| < \e, \ \ \ \mbox{and} \ \ \  \|S - T\| < \e.
\end{equation*}
Among other results, a uniformly convex and uniformly smooth complex Banach space satisfies the BPBp-$\nu$ (see \cite[Corollary 7]{KLM1}). In particular, so do complex Hilbert spaces and complex $L_p$-spaces with $1 < p < \infty$. Actually, a real Hilbert space and an $L_1(\mu)$ space for every measure $\mu$ also satisfy the BPBp-$\nu$ (see \cite[Theorem 3.2]{KLM} and \cite[Theorem 9]{KLM1} (or \cite[Theorem 9]{F}), respectively). However, every separable infinite dimensional Banach space can be renormed to fail the BPBp-$\nu$ (\cite[Theorem 17]{KLM1}). 

Similarly to the BPBpp, we are interested in studying a stronger property than the BPBp-$\nu$ for classical operators on complex Hilbert spaces. To be more precise, we introduce the Bishop-Phelps-Bollob\'as point property for numerical radius: we say that a Banach space $X$ has the {\it Bishop-Phelps-Bollob\'as point property for numerical radius} (the BPBpp-$\nu$, for short) if given $\e > 0$, there is $\eta(\e) > 0$ such that whenever $T \in \mathcal{L}(X)$ with $\nu(T) = 1$ and $(x, x^*) \in S_X \times S_{X^*}$ with $x^*(x) = 1$ satisfy $|x^*(T x )| >  1 - \eta(\e)$, there is a new operator $S \in \mathcal{L}(X)$ with $\nu(S) = 1$ such that 
\begin{equation*}
|x^*(S x)| = 1 \ \ \ \mbox{and} \ \ \ \|S - T\| < \e.
\end{equation*}
It was recently discovered that if the numerical index of a Banach space $X$ (the numerical index of $X$ is defined by $n(X) = \inf \{ \nu(T): T \in \mathcal{L}(X), \|T\| = 1 \}$) is one and $X$ satisfies the BPBpp-$\nu$, then X must be one-dimensional \cite{DKLM}. On the other hand, as we have mentioned before, $L_1(\mu)$ satisfies the BPBp-$\nu$ for every measure $\mu$ (see \cite[Theorem 9]{KLM1}). Thus, since the numerical index of $L_1(\mu))$ is one, $L_1(\mu)$ is an example of a Banach space which has the BPBp-$\nu$ but not the BPBpp-$\nu$.

Let us now give the contents of this paper. In Section 2, we recall some properties of a resolution of the identity on a complex Hilbert space, and show a technical result which allows us to transfer the BPBp-$\nu$ (resp. the BPBp) to the BPBpp-$\nu$ (resp. the BPBpp). In Section 3, we study the Bishop-Phelps-Bollob\'as point property for some classes of operators defined on a complex Hilbert space as self-adjoint, anti-symmetric, unitary, normal, compact, and Schatten-von Neumann. As a consequence of these results and their proofs, we get the analogous for positive, positive Schatten-von Neumann, compact positive, self-adjoint Schatten-von Neumann, and normal Schatten-von Neumann operators. Finally, in Section 4, we consider similar problems for the Bishop-Phelps-Bollob\'as point property for numerical radius.

\section{Preliminaries}

In this section we show some technical results, which we need in discussing the problems that appear in sections 3 and 4. We begin with giving the definition of the BPBp (and BPBp-$\nu$) for a class of operators $\mathcal{A}$ and recall the definition and some properties of the Schatten-von Neumann classes and some basic notation and results from spectral measure. After this, we apply the fact that Hilbert spaces have transitive norms in order to transfer the BPBp-$\nu$ (resp. the BPBp) to the BPBpp-$\nu$ (resp. the BPBpp).

The definition of the BPBp (resp. the BPBpp) for compact operators already appeared in \cite[Definition 1.4]{DGMM} (resp. \cite[Definition 5.1]{DKKLM}) and the definition of BPBp-$\nu$ for $\mathcal{A} \subset \mathcal{L}(X)$ appeared in \cite[Definition 2.1]{AcoFakSole}. 
Next, we state, for a Hilbert space $H$, the definitions of the BPBp (and BPBpp) for $\mathcal{A} \subset \mathcal{L}(H)$ and the BPBp-$\nu$ (and BPBpp-$\nu$) for $\mathcal{A} \subset \mathcal{L}(H)$ that we are working with in this paper.

\begin{definition} \label{def} Let $H$ be a Hilbert space and $\mathcal{A} \subset \mathcal{L}(H)$. 
	
	\begin{enumerate}[label=(\alph*)] 
		\item We say that $H$ has the BPBp for $\mathcal{A}$ if given $\e > 0$, there is $\eta(\e) > 0$ such that whenever $T \in \mathcal{A}$ with $\| T\|  = 1$ and $x_0 \in S_H$ satisfy 
		\begin{equation*}
		\| T x_0\| > 1 - \eta(\e),
		 \end{equation*}
		there are $S \in \mathcal{A}$ with $\| S \|  = 1$ and $x_1 \in S_H$ such that 
		\begin{equation*}
		\|S x_1\| = 1, \ \ \ \|x_1 - x_0\| < \e, \ \ \ \mbox{and} \ \ \  \|S - T\| < \e. 
		\end{equation*}
		If $x_1 = x_0$, then we say $H$ has the BPBpp for $\mathcal{A}$.
		\vspace{0.1cm}
		\item We say that $H$ has the BPBp-$\nu$ for $\mathcal{A}$ if given $\e > 0$, there is $\eta(\e) > 0$ such that whenever $T \in \mathcal{A}$ with $\nu(T) = 1$ and $x_0 \in S_H$ satisfy 
		\begin{equation*}
		|\langle T x_0, x_0 \rangle| > 1 - \eta(\e),
		\end{equation*}
		 there are $S \in \mathcal{A}$ with $\nu(S) = 1$ and $x_1 \in S_H$ such that 
		\begin{equation*}
		| \langle S x_1, x_1 \rangle| = 1, \ \ \ \|x_1-x_0\|<\e, \ \ \ \mbox{and} \ \ \  \|S - T\| < \e. 
		\end{equation*}
		If $x_1 = x_0$, then we say $H$ has the BPBpp-$\nu$ for $\mathcal{A}$.
	\end{enumerate}
\end{definition}

Let $H$ be a {\bf complex} Hilbert space. For a compact operator $T \neq 0$ on $H$, the operator $|T|$ has the spectral representation 
\begin{equation} \label{spectral}
|T| = \sum_{j=1}^{n_0} \lambda_j \langle \cdot, x_j \rangle x_j, 
\end{equation} 
where $n_0 \in \mathbb{N} \cup \{\infty\}$, $\{\lambda_j\}$ is the sequence of non-zero eigenvalues of $|T|$ (arranged in decreasing order and counted according to their multiplicities), and $\{x_j\}$ is the corresponding orthonormal sequence of eigenvectors. For $1 \leq p < \infty$, the {\it Schatten-von Neumann class} $S_p (H)$ consists of all compact operators $T$ with
\begin{equation*} 
\sigma_p (T) =  \left(\sum_{j=1}^{\infty} \lambda_j^p \right)^{1/p} < \infty.
\end{equation*}
$S_p (H)$ is a Banach space endowed with the Schatten $p$-norm $\sigma_p( \cdot )$. The elements of $S_p(H)$ are called {\it Schatten-von Neumann operators}. We define $S_{\infty}(H)$ to be simply $\mathcal{L}(H)$. It is well-known that the Schatten $p$-norm has the monotonicity property: for $1 \leq p \leq p' \leq \infty$,
\begin{equation}\label{monotonicity}
\|T\|=\sigma_{\infty} (T) \leq \sigma_{p'} (T) \leq \sigma_p (T) \leq \sigma_1 (T).
\end{equation}
In Theorem \ref{propsec1}, we prove not only that $H$ has the BPBpp for Schatten-von Neumann operators
but also that a given Schatten-von Neumann operator can be approximated by some operator of the same class in the Schatten $p$-norm (see \cite[Theorem 4.1]{GLM}). To do so, we need the following generalization of the H\"{o}lder inequality. Suppose that $1 \leq r, s,t \leq \infty$, $t^{-1} = r^{-1} + s^{-1}$, $R \in S_r (H)$, and $S \in S_s (H)$. Then $RS \in S_t (H)$ and $\sigma_t (RS) \leq \sigma_r (R) \sigma_s (S)$ (see, for example, \cite[Theorem 2.3.10]{JRR}).

Let $\mathfrak{M}$ be a $\sigma$-algebra in a set $\Omega$. A \emph{resolution of the identity} (on $\mathfrak{M}$) is a mapping $E : \mathfrak{M} \rightarrow \mathcal{L}(H)$
with the following properties: 
\begin{enumerate}
\item $E(\emptyset) =0, E(\Omega) = \Id_H$.
\item Each $E(\omega)$ is a self-adjoint projection.
\item $E(\omega' \cap \omega'') = E(\omega')E(\omega'')$.
\item If $\omega' \cap \omega'' =\emptyset$, then $E(\omega' \cup \omega'') = E(\omega') + E(\omega'')$. 
\item For every $x \in H$ and $y \in H$, the set function $E_{x,y} (\omega) = \langle E(\omega)x, y\rangle$ is a complex measure on $\mathfrak{M}$
\end{enumerate}
(see, for example, \cite[Definition 12.17]{RUD}). Recall that if $T \in \mathcal{L}(H)$ is normal, then there exists a unique resolution of the identity $E$ on the Borel subsets of $\sigma(T)$, which satisfies 
\[
T = \int_{\sigma(T)} z \, dE(z).
\]
Furthermore, every projection $E(\omega)$ commutes with every $S \in \mathcal{L}(H)$ which commutes with $T$. Moreover, with the same hypothesis, if  $f : \sigma(T) \rightarrow \mathbb{C}$ is a bounded Borel function, $\delta >0$, $B(\delta)$ denotes the closed disk centered at the origin with radius $r>0$ in $\mathbb{C}$,
\begin{equation*} 
	N_1 = \int_{\sigma(T)\setminus B(\delta)} f(z) \, dE(z) \ \ \ \mbox{and} \ \ \ 
	N_2 = \int_{\sigma(T) \cap B(\delta)} z\,dE(z),
\end{equation*}
then
	\begin{enumerate}
		\item $\ran E(\sigma(T)\setminus B(\delta) \subset \overline {\ran T}$.
		\item $\ran N_1 \subset \ran E(\sigma(T) \setminus B(\delta))$ and $\ker N_1 \supset \ran E(\sigma(T)\cap B(\delta))$.
		In particular, if $|f(z)|>0$ for all $z \in \sigma(T) \setminus B(\delta)$, then  $\overline{ \ran N_1} = \ran E(\sigma(T)\setminus B(\delta))$.		
		\item $\ran N_2 \subset \ran E(\sigma(T)\cap B(\delta))$ and $\ker N_2 \supset \ran E(\sigma(T)\setminus B(\delta)).$ 
	\end{enumerate} 
This can be found, for example, in \cite[Lemma 2.4]{CD}. Also, we denote by $f(T)$ the operator
\begin{equation*}
\int_{\sigma(T)} f(z) E(z),	
\end{equation*}
where $f$ is a bounded Borel function on $\sigma(T)$. Moreover, we need the following result.

\begin{lemma}\label{lem2}\cite[Proposition 4.1]{CON}.
	 If $T$ is a normal operator and $T = \int z \, dE(z)$, then $T$ is compact if and only if for every $\eps >0$, $E (\{ z : |z| > \eps \})$ has finite rank. 
	\end{lemma}

In order to prove Theorem \ref{maintheorem}, we need the following two lemmas. The first one says the well-known fact that Hilbert spaces have transitive norms. If $T \in \mathcal{L}(H)$, we denote by $T^*$ the adjoint operator of $T$.

\begin{lemma}{(\cite[Lemma 2.2]{AMS})} \label{maria} Let $H$ be a (real or complex) Hilbert space. Given $x$ and $y$ in $S_H$, there is a surjective isometry $R \in \mathcal{L}(H)$ such that 
\begin{equation*}	
R x = y \ \ \ \mbox{and} \ \ \  \|R - \Id_H\| = \|x - y\|.
\end{equation*}
\end{lemma}

\begin{lemma} \label{ref} Let $H$ be a complex Hilbert space. Given $x, y \in S_H$, consider the surjective isometry $R \in \mathcal{L}(H)$ from Lemma \ref{maria}. Define $\mathcal{R}_{x, y}: \mathcal{L}(H) \longrightarrow \mathcal{L}(H)$ by $\mathcal{R}_{x, y} (T) := R^* \circ T \circ R$ for $T \in \mathcal{L}(H)$. Then, for every $T \in \mathcal{L}(H)$, we have
\begin{itemize}
\item[(i)] $\nu(T) = \nu (\mathcal{R}_{x, y}(T))$ and $\|T\| = \|\mathcal{R}_{x, y}(T)\|$.
\item[(ii)] $\langle Ty, y \rangle = \langle \mathcal{R}_{x, y}(T)(x), x \rangle$ and $\|T(y)\| = \|\mathcal{R}_{x, y}(T)(x)\|$.
\item[(iii)] $\| \mathcal{R}_{x, y}(T) - T\| \leq 2 \|x - y\| \|T\|$.
\end{itemize}
\end{lemma}

\begin{proof} (i) is clear, because $R$ is a surjective isometry. For (ii), note that
	\begin{equation*}
\langle \mathcal{R}_{x, y}(T)(x), x \rangle = \langle (R^* \circ T \circ R)(x), x \rangle = \langle (T \circ R)(x), R x \rangle = \langle T y , y \rangle
	\end{equation*}
and $\| \mathcal{R}_{x, y}(T)(x)\| = \|(T \circ R)(x)\| = \|T y \|$. Finally, (iii) holds since
\begin{eqnarray*}
\| \mathcal{R}_{x, y}(T) - T\| &=& \| R^* \circ T \circ R -T\| \\
&\leq& \| R^* \circ T \circ R - T \circ R\| + \|T \circ R - T\| \\
&\leq& \| R^* - Id_H\| \|T \circ R\| + \|R - Id_H\| \|T\| \\
&=& \|x - y\| \|T\| + \|x - y\| \|T\|.	
\end{eqnarray*}	
	
\end{proof}

Now we are ready to prove the desired theorem that we will use in the next sections. 

\begin{theorem} \label{maintheorem} Let $H$ be a complex Hilbert space. Let $\mathcal{A} \subset \mathcal{L}(H)$ be such that $H$ has the BPBp-$\nu$ (resp. the BPBp) for $\mathcal{A}$ and suppose that $\mathcal{R}_{x, y} \mathcal{A} \subset \mathcal{A}$ for every $x, y \in S_H$, where $\mathcal{R}_{x, y}$ is defined as in Lemma \ref{ref}. Then, $H$ has the BPBpp-$\nu$ (resp. the BPBpp) for $\mathcal{A}$.	
\end{theorem}

\begin{proof} We give a proof for numerical radius. Let $\e > 0$ be given. By hypothesis, we can consider $\eta(\e) > 0$ such that whenever $T \in \mathcal{A}$ with $\nu(T) = 1$ and $x_0 \in S_H$ satisfy 
	\begin{equation*}
	| \langle Tx_0, x_0 \rangle| > 1 - \eta(\e),	
	\end{equation*}
there are $\widetilde{S} \in \mathcal{A}$ with $\nu(\widetilde{S}) = 1$ and $x_1 \in S_H$ such that	
\begin{equation*}
| \langle \widetilde{S} x_1 , x_1 \rangle| = 1, \ \ \ \|x_1 - x_0\| < \e \ \ \ \mbox{and} \ \ \ \|\widetilde{S} - T\| < \e.	
\end{equation*}	

Define $S := \mathcal{R}_{x_0, x_1}(\widetilde{S})$.  By hypothesis, $S \in \mathcal{A}$ and $\|\widetilde{S}\| \leq 2 \nu(\widetilde{S}) = 2$, because the numerical index of a complex Hilbert space $H$ is 1/2. It follows from Lemma \ref{ref} that $| \langle S x_0, x_0 \rangle| = 1 = \nu(S)$ and
\begin{equation*}
\|S - T\| \leq \| S - \widetilde{S}\| + \| \widetilde{S} - T\| < 4 \e + \e = 5 \e.	
\end{equation*}
\end{proof}

\section{The Bishop-Phelps-Bollob\'{a}s point property for $\mathcal{A} \subset \mathcal{L}(H)$}

In this section, we prove that a complex Hilbert $H$ satisfies the BPBpp for some classical operators defined on $H$. It worth mentioning that in Theorem \ref{propsec1} the items \ref{pp_operators}, \ref{pp_cpt_operators}, and \ref{pp_SvN_operators} can be obtained from \cite[Corollary 2.3]{ABGM} (see also \cite{DKKLM}) combined with Theorem \ref{maintheorem} due to the uniform convexity of a Hilbert space and that items \ref{pp_selfadj_operators} and \ref{pp_normal_operators} can be shown, with the aid of Theorem \ref{maintheorem}, by using the facts from \cite{GLM} and \cite{CD}, respectively. Nevertheless, in what follows, by elaborating a spectral measure technique, we shall give a proof using symbolic calculus  which will cover all these results and even more cases (see Proposition \ref{referee1}).

\begin{theorem}\label{propsec1} Let $H$ be a complex Hilbert space. Then,
\begin{enumerate}[label=(\alph*)]
\item $H$ has the BPBpp for operators.	\label{pp_operators}
\item $H$ has the BPBpp for self-adjoint operators.	\label{pp_selfadj_operators}
\item $H$ has the BPBpp for compact self-adjoint operators.	\label{pp_cpt_selfadj_operators}
\item $H$ has the BPBpp for anti-symmetric operators.	\label{pp_antisym_operators}
\item $H$ has the BPBpp for unitary operators.	\label{pp_unitary_operators}
\item $H$ has the BPBpp for normal operators.	\label{pp_normal_operators}
\item $H$ has the BPBpp for compact normal operators.	\label{pp_cpt_normal_operators}
\item $H$ has the BPBpp for compact operators. 	\label{pp_cpt_operators}
\item $H$ has the BPBpp for Schatten-von Neumann operators.	\label{pp_SvN_operators}
\end{enumerate}
\end{theorem}

\begin{proof} 

Let $0< \eps < 1$ and $T$ be a {\it positive} operator with norm $1$ and $\| T x_0 \| > 1 -\eps^2 / 4$ for some $x_0 \in S_H$. Let $y_0 \in S_H$ be such that $\langle T x_0, y_0 \rangle > 1-\eps^2 / 4$. Since $T \geq 0$, we have that $T$ is self-adjoint and $\sigma(T) \subset [0, \infty)$; hence it follows from \cite[Theorem 2.1]{GLM} that there are a self-adjoint operator $R \in \mathcal{L}(H)$ with $\|R\| = 1$ and a vector $x_1 \in S_H$ such that 
\begin{equation*}
\langle Rx_1, x_1 \rangle = 1, \ \ \|R- T \| <\eps, \ \ \|x_0-x_1 \| <4\sqrt{\eps}, \ \ \mbox{and} \ \ \|y_0 -  x_1 \| < 4\sqrt{\eps}. 
\end{equation*}
Indeed, $R$ and $x_1$ are constructed explicitly as 
\begin{equation*}
R=E(A)  + \int_B z \, dE(z) \quad\text{and}\quad x_1 = \frac{E(A)x_0}{\|E(A)x_0\|},
\end{equation*}
where $E$ is the spectral measure of $(\sigma(T), \mathcal{B}(\sigma(T)), H)$,
\begin{equation*}
A = \{ z \in \sigma(T) : z > 1-\eps\}, \ \ \mbox{and} \ \  B = \{z \in \sigma (T) : 0\leq z \leq 1-\eps\}
\end{equation*}
(notice that, since $T \geq 0$, $A_{-} = \{z \in \sigma(T): z < -1 + \e \} = \emptyset$ and then $y_1 = x_1$ in \cite[Theorem 2.1]{GLM}). Observe that the operator $R$ can be rewritten as $R=T f(T)$, where $f:[0,1]\rightarrow [0,\infty)$ is defined as 
\begin{equation} \label{f}
f(t) =
\begin{cases}
1 \,\qquad t \in [0,1-\eps], \\
\frac{1}{t} \qquad t \in (1-\eps,1]
\end{cases}
\end{equation}
and $f(T)$ denotes the symbolic calculus for $T$. With these considerations, we can start our proof.

For a general operator $T \in \mathcal{L} (H)$ with $\|T\| = 1$, we suppose that $\|T x_0 \| > 1-\eps^2 /4$ for some $x_0 \in S_X$. Take $y_0 \in S_H$ so that $\langle Tx_0, y_0\rangle > 1 -\eps^2 /4$. Consider the factorization $T = U|T|$, where $U$ is a partial isometry. Then,  
\[
\left\langle |T| x_0, \frac{U^* y_0}{\|U^* y_0\|} \right\rangle \geq \langle |T| x_0, U^* y_0 \rangle = \langle U|T| x_0, y_0\rangle  > 1- \frac{\eps^2}{4}.
\]
By using the first part of the proof, we	 
consider the operator $|T| f(|T|),$ where $f$ is defined in \eqref{f}, and $x_1 \in S_H$ satisfying
\begin{enumerate}[label=(\roman*)]
\item $\| |T| f(|T|) \| = \langle |T| f(|T|) x_1, x_1 \rangle = 1$, 
\item $\| |T| f(|T|) - |T| \| < \eps$, 
\item $\|x_0-x_1\| < 4\sqrt{\eps}$, and  
\item $\|U^* y_0/\|U^* y_0\| -x_1\| <4 \sqrt{\eps}$.
\end{enumerate}
Now consider $S := U |T| f(|T|) = T f(|T|)$ and notice that $x_1 \in \ran  E(A) \subseteq \overline{\ran |T|}=(\ker|T|)^{\perp}$ which implies that $U^* U x_1 = x_1$. We then have
\[
\langle S x_1, U x_1 \rangle = \langle |T| f(|T|) x_1, x_1 \rangle =1
\]
which implies $\|S \| = \| Sx_1 \| = 1$. Moreover, by (ii), $\|S-T\|=\|U |T| f(|T|) - U|T|\| < \eps$. This proves that $H$ has the BPBp for operators. By Theorem \ref{maintheorem}, $H$ has the BPBpp for operators and we get \ref{pp_operators}.

Next, we claim that $S$ defined as above is self-adjoint, normal, compact, and Schatten-von Neumann, whenever $T$ is self-adjoint, normal, compact, and Schatten-von Neumann, respectively. We first show it for normal operators. If $T$ is normal, then the partial isometry $U$, which is actually unitary in this case, can be chosen so that $U |T| = |T| U$ and this implies that $U g(|T|) = g(|T|) U$ for every bounded Borel function $g$ (see, for example, \cite[section 12.24]{RUD}). Thus, 
\begin{eqnarray*}
S^* S &=& (f(|T|) T^* )(T f(|T|)) \\ 
&=& (f(|T|) T)(T^* f(|T|)) \\
&=& (f(|T|) U |T|)(f(|T|) T)^* \\
&=& (U |T| f(|T|))(U |T| f(|T|))^* 
= S S^*,
\end{eqnarray*}
so $S$ is normal. An analogous argument proves that $S$ is self-adjoint when $T$ is self-adjoint. 
Since the compact and Schatten-von Neumann operators are operator ideals, our claim is achieved.
This proves \ref{pp_selfadj_operators}, \ref{pp_normal_operators}, \ref{pp_cpt_operators}, and \ref{pp_SvN_operators} and also \ref{pp_cpt_selfadj_operators} and \ref{pp_cpt_normal_operators}. Notice that \ref{pp_antisym_operators} is just a consequence of \ref{pp_selfadj_operators} and that \ref{pp_unitary_operators} is trivial.
	
Finally, we give a result that we can approximate a Schatten-von Neumann operator $T \in S_p (H)$ not only in the operator norm but also in Schatten $p$-norm. Suppose that $\|T x_0 \| > 1-\eps^2 /4$ for some $x_0 \in S_X$. By \cite[Theorem 4.1]{GLM}, ${S} = U |T| f(|T|) \in S_p (H)$ and $x_1 \in S_H$ satisfy 
	\begin{equation*}
		\|{S}\| = \|{S} x_1\| = 1, \ \ \ \|x_1 - x_0\| < \beta(\e), \ \ \ \mbox{and} \ \ \ \sigma_p ({S}-T) < 2\eps M,
	\end{equation*}
	where $\sigma_p (T) \leq M$ and $T=U|T|$ is the polar decomposition of $T$. 
	By Lemma \ref{maria}, there is a surjective isometry $R$ such that $R(x_0) = x_1$ and $\| R- \Id_H\| = \|x_0 - x_1\| < \beta(\e)$. Define $\widetilde{S} = {S} \circ R$. Since Schatten norms are isometrically invariant, $\sigma_p (\widetilde{S}) = \sigma_p ({S}\circ R) = \sigma_p({S})$, and $\widetilde{S} \in S_p (H)$. Moreover, $\| \widetilde{S} x_0 \| = \|({S} \circ R) (x_0) \| = \|{S} x_1 \| = 1$. Since $\| \widetilde{S} \| = \|{S} \circ R\| = \|{S}\|$, we obtain that $\| \widetilde{S} \| = \| \widetilde{S} x_0 \| = 1$. Finally, by using H\"{o}lder's inequality, we get that
	\begin{eqnarray*} 
		\sigma_p (T - \widetilde{S}) = \sigma_p (T - {S} \circ R)
		&\leq& \sigma_p (T - {S}) + \sigma_p({S} - {S} \circ R) \\ 
		&\leq& \sigma_p (T-{S}) + \sigma_p ({S}(\Id_H - R)) \\
		&\leq& 2 \eps M + \sigma_p ({S}) \| \Id_H - R \|  \\ 
		&\leq& 2 \eps M + (\sigma_p (T) + \sigma_p ({S}-T)) \beta(\eps) \\
		&<& 2\eps M + (1 + 2\eps )M \beta(\eps).
	\end{eqnarray*}	
	Notice from the monotonicity property \eqref{monotonicity} of Schatten $p$-norm that $\|T-\widetilde{S} \| < 2\eps M + (1 + 2\eps )M \beta(\eps)$ automatically. 
\end{proof}

Let us notice the following about Theorem \ref{propsec1}. From the first part of the proof, we have that when $T \in \mathcal{L}(H)$ is a positive operator with norm $1$ and $\| T x_0 \| > 1 - \eps^2/ 4$ for some $x_0 \in S_H$, there exists a self-adjoint operator $R=T f(T)  \in \mathcal{L}(H)$ which attains the norm at $x_1 \in S_H$ with $\|x_1 - x_0 \| < 4\sqrt{\eps}$ and satisfies $\|R- T\| < \eps$. 
Note that 
\begin{align*}
\langle Rx, x \rangle = \langle E(A) x, x \rangle + \left\langle \int_B z \, dE(z) x, x \right\rangle,
\end{align*} 
for every $x \in H$, where $A = \{ z \in \sigma(T) : z > 1-\eps \}$ and $B = \{z \in \sigma(T) : 0 \leq z \leq 1-\eps\}$. Since $E(A)$ is a self-adjoint projection (so, the set function $E_{x,x}$ is a positive measure on Borel subsets of $\sigma(T)$), 
\[
\langle E(A) x, x \rangle = \| E(A) x \|^2 \geq 0 \quad \text{and}\quad  \int_B z \, dE_{x,x} (z) \geq 0. 
\]
It follows that $R$ is a positive operator. Therefore, if we start with a positive operator (resp. positive Schatten-von Neumann operator), then we end up with another positive operator (resp. positive Schatten-von Neumann operator). It is clear that the operator $R=T f(T)$ above is compact whenever $T$ is compact positive and that $S = U|T|f(|T|)$ is self-adjoint and normal whenever $T$ is self-adjoint and normal, respectively.

Thus, to sum it up, we have the following result.

\begin{prop} \label{referee1} Let $H$ be a complex Hilbert space.
\begin{enumerate}[label=(\alph*)]
\item $H$ has the BPBpp for positive operators.
\item $H$ has the BPBpp for positive Schatten-von Neumann operators. 
\item $H$ has the BPBpp for compact positive operators.  
\item $H$ has the BPBpp for self-adjoint Schatten-von Neumann operators. 
\item $H$ has the BPBpp for normal Schatten-von Neumann operators. 
\end{enumerate} 
\end{prop}

\section{The Bishop-Phelps-Bollob\'{a}s point property for numerical radius for $\mathcal{A} \subset \mathcal{L}(H)$}

In this section, we consider the analogue of Theorem \ref{propsec1} and Proposition \ref{referee1} for the BPBpp-$\nu$.

\begin{theorem}\label{propsec2} Let $H$ be a complex Hilbert space. Then,
\begin{enumerate}[label=(\alph*)]
\item $H$ has the BPBpp-$\nu$ for operators.	\label{nu_operators}
\item $H$ has the BPBpp-$\nu$ for self-adjoint operators.	\label{nu_selfadj_operators}
\item $H$ has the BPBpp-$\nu$ for compact self-adjoint operators \label{nu_cpt_selfadj_operators}
\item $H$ has the BPBpp-$\nu$ for anti-symmetric operators.	\label{nu_antisym_operators}
\item $H$ has the BPBpp-$\nu$ for unitary operators.	\label{nu_unitary_operators}
\item $H$ has the BPBpp-$\nu$ for normal operators.	\label{nu_normal_operators}
\item $H$ has the BPBpp-$\nu$ for compact normal operators.	\label{nu_cpt_normal_operators}
\item $H$ has the BPBpp-$\nu$ for compact operators.	\label{nu_cpt_operators}
\item $H$ has the BPBpp-$\nu$ for Schatten-von Neumann operators.	\label{nu_SvN_operators}
\end{enumerate}
\end{theorem}

\begin{proof} 
Let $\e \in (0, 1)$ be given. By \cite[Corollary 7]{KLM1}, there exists $\eps \mapsto \eta(\eps)$ such that whenever $T \in \mathcal{L} (H)$ with $\nu (T) =1$ and $x_0 \in S_H$ satisfy 
\begin{equation}\label{prop4.1.1}
		|\langle Tx_0, x_0 \rangle| > 1 - \min\left\{ \eps, \eta\left(\eps\right) \right\}
	\end{equation}
there are $\widetilde{S} \in \mathcal{L}(H)$ with $\nu(\widetilde{S}) = 1$ and $x_{\infty} \in S_H$ such that
	\begin{equation*}
		|\langle \widetilde{S}x_{\infty}, x_{\infty} \rangle| = 1, \ \ \ \|x_{\infty} - x_0\| < \eps, \ \ \  \mbox{and} \ \ \ \|\widetilde{S} - T\| < \eps
	\end{equation*}
Following the proofs of \cite[Proposition 4 and Proposition 6]{KLM1}, one can observe that the operator $\widetilde{S}$ is constructed from a limit of a sequence of operators $\{T_n\}$, where 
\begin{equation}\label{Kn}
T_n = T + K_n \quad \text{and} \quad K_n = 
\alpha_1 \left(\frac{\eps}{4}\right) \langle \ \cdot \ , x_1 \rangle x_1 + \cdots \alpha_n \left(\frac{\eps}{4}\right)^n  \langle \ \cdot \ , x_n \rangle x_n
\end{equation}
for some $\alpha_1, \dots, \alpha_n$ in $S_{\mathbb{C}}$ and vectors $x_1, \dots, x_n$ in $S_{H}$. 
At the same time, the vector $x_{\infty}$ is obtained as a limit of a sequence of vectors $\{x_n\}$ satisfying 
\begin{equation}\label{Tnxn}
\lim_n \nu (T_n) = \lim_n |\langle T_n x_n, x_n \rangle|. 
\end{equation} 
It follows that $\tilde{S}$ is compact whenever $T$ is compact. 
Thus, \ref{nu_operators} and \ref{nu_cpt_operators} hold by applying Theorem \ref{maintheorem}. 

To observe \ref{nu_selfadj_operators} and \ref{nu_cpt_selfadj_operators}, we assume that the above $T \in \mathcal{L}(H)$ is a self-adjoint operator (resp. compact self-adjoint operator). 
	Since $\langle Tx_0, x_0 \rangle \in \R$, we may assume that $\langle Tx_0, x_0 \rangle > 0$ (otherwise, we would work with $-T$). For some $\theta \in \R$, we have
	\begin{equation*}	
		\langle \widetilde{S}x_{\infty}, x_{\infty} \rangle = e^{i \theta} |\langle \widetilde{S}x_{\infty}, x_{\infty} \rangle| = e^{i \theta} \in S_{\C}.
	\end{equation*}
	Set $r := \langle Tx_0, x_0 \rangle \in \R^+$. We have that $\langle (e^{-i\theta}\widetilde{S})x_{\infty}, x_{\infty} \rangle = 1$ and	 that
	\begin{equation*}
		|e^{i \theta} - r| = |\langle \widetilde{S}x_{\infty}, x_{\infty} \rangle - \langle Tx_0, x_0 \rangle| \leq \|\widetilde{S} - T\| + 2\|x_{\infty} - x_0\| < 3\eps.
	\end{equation*}
	So, 
	\begin{equation*}
		|e^{i \theta} - 1| \leq |e^{i \theta} - r| + |r - 1| < 4\eps.
	\end{equation*}
Since $\|\widetilde{S}\| \leq 2 \nu(\widetilde{S}) = 2$, we get
	\begin{equation*}
		\|\widetilde{S} - (e^{-i\theta} \widetilde{S})\| \leq |1 - e^{-i \theta}| \|\widetilde{S}\| \leq 2 | e^{i \theta}- 1|  < 8\eps,
	\end{equation*}
	which implies that
	\begin{equation*}
		\|(e^{-i \theta} \widetilde{S}) - T\| \leq \|(e^{-i \theta} \widetilde{S}) - \widetilde{S}\| + \|\widetilde{S} - T\| < 9\eps.
	\end{equation*}
	Note that we just proved that the operator (resp. compact operator) $S' := (e^{-i \theta} \widetilde{S}) \in \mathcal{L}(H)$ satisfies 
	\begin{equation*}
		\nu(S') = \re \ \langle S'x_{\infty}, x_{\infty} \rangle = 1 \ \ \ \mbox{with} \ \ \ \|x_{\infty} - x_0\| < \e \ \ \ \mbox{and} \ \ \  \|S' - T\| < 9\e. 
	\end{equation*}
	Now define $S := \frac{S' + (S')^*}{2} \in \mathcal{L}(H)$. Then $S$ is self-adjoint (resp. compact self-adjoint), $\nu(S) = \|S\| \leq 1$, and	
	\begin{equation*}
		|\langle Sx_{\infty}, x_{\infty} \rangle| = \left| \frac{1}{2} \langle S'x_{\infty}, x_{\infty} \rangle + \frac{1}{2} \overline{\langle S'x_{\infty}, x_{\infty} \rangle} \right| = \re \ \langle S'x_{\infty}, x_{\infty} \rangle = 1. 
	\end{equation*}
	Hence, $\nu(S) = |\langle Sx_{\infty}, x_{\infty} \rangle| = 1$. Finally, since $T = T^*$, we have
	\begin{equation*}
		\|S - T\| \leq \frac{1}{2} \|S' - T \| + \frac{1}{2} \|(S')^* - T\| \\
		= \frac{1}{2} \|S' - T \| + \frac{1}{2} \|(S')^* - T^*\| < 9\e,
	\end{equation*}
which completes the proof of \ref{nu_selfadj_operators} and \ref{nu_cpt_selfadj_operators} due to Theorem \ref{maintheorem}. Note that \ref{nu_antisym_operators} follows directly from \ref{nu_selfadj_operators}. 

Next, we prove \ref{nu_SvN_operators} by approximating a given Schatten-von Neumann operator by some operator in the same class in the $p$-Schatten norm, which will imply that $H$ has the BPBpp for Schatten-von Neumann operators due to the monotonicity \hypersetup{linkcolor=blue}\eqref{monotonicity}. Indeed, suppose that $T \in S_p(H)$ with $\nu(T) = 1$ satisfies \eqref{prop4.1.1} with the same $\eps \mapsto \eta(\eps)$ for some $x_0 \in S_H$. 
	Note that the finite rank operator $K_n$ in \eqref{Kn} belongs to $S_p (H)$, so $T_n \in S_p (H)$. Also, $\sigma_p (T_{n+1} - T_n) =\eps^{n+1}/4^{n+1}$ for every $n \in \mathbb{N}$.
This shows that $\{T_n\}$ is a Cauchy sequence in $S_p (H)$, so $T_n \rightarrow T_{\infty}$ for some $T_{\infty} \in S_p (H)$ (and hence $\|T_n - T_{\infty}\| \rightarrow 0$ as well). Note that $\sigma_p (T_{\infty} - T) \leq \eps/(4-\eps) < \eps$ and from \eqref{Tnxn} that 
\[
\nu (T_{\infty}) = \lim_{n \rightarrow \infty} \nu (T_n) = \lim_n |\langle T_n x_n, x_n \rangle| = |\langle T_{\infty} x_{\infty}, x_{\infty}\rangle|.
\]

	Since 
	\[
	|1 - v(T_{\infty})| = |\nu(T) - v(T_{\infty})| \leq \|T-T_{\infty}\| < \eps,
	\]
	we have $\nu (T_{\infty}) > 1 - \eps >0$. We define $\tilde{S} = \frac{1}{v(T_{\infty})} T_{\infty} \in S_p(H)$, then $v (\tilde{S}) = |\langle \tilde{S} x_{\infty}, x_{\infty} \rangle| = 1$. 
	Since
	\begin{align*}
		\sigma_p (\tilde{S} - T_{\infty} ) = \sigma_p \left( \left( \frac{1- v (T_{\infty})}{v(T_{\infty} )} \right) T_{\infty} \right) 
		&= \left| \frac{1-v(T_{\infty} )}{v(T_{\infty})}\right| \sigma_p (T_{\infty}) \\ 
		&\leq \left(\frac{\eps}{1-\eps}\right) \sigma_p (T_{\infty}) \\
		&\leq \left(\frac{\eps}{1-\eps}\right)( \sigma_p (T) + \sigma_p (T_{\infty} - T) ) \\
		&< \left(\frac{\eps}{1-\eps}\right)( M + \eps), 
	\end{align*}  
	where $\sigma_p (T) \leq M$ for some $M > 0$,
	we obtain 
	\begin{align*}
		\sigma_p (\tilde{S} - T) \leq \sigma_p (\tilde{S} - T_{\infty} ) + \sigma_p (T_{\infty} - T) < \left(\frac{\eps}{1-\eps}\right)( M + \eps) + \eps.
	\end{align*}  
Applying Theorem \ref{maintheorem}, we finish the proof of \ref{nu_SvN_operators}.

We prove \ref{nu_unitary_operators} directly. Let $\e \in (0, 1)$ be given and $T \in \mathcal{L}(H)$ be unitary with $\nu(T) = 1$. Now pick $x_0 \in S_H$ be such that $\left| \langle Tx_0, x_0 \rangle \right|> 1 - \frac{\e^2}{2}$. Let $\theta\in\mathbb{R}$ such that $\langle Tx_0, x_0 \rangle $ = $e^{i\theta} \left|\langle Tx_0, x_0 \rangle \right| $. Then
	\begin{eqnarray*}
		\|T x_0 - e^{i\theta} x_0\|^2 &=& \|T x_0\|^2 + \|x_0\|^2 - 2 \re \ \langle Tx_0, e^{i\theta} x_0 \rangle \\
		&=&  2 - 2 \left|\langle Tx_0, x_0 \rangle \right| \\
		&<& 2-2\left( 1-\frac{\eps^2}{2}\right) = \e^2.
	\end{eqnarray*}	
So, $\|T x_0  - e^{i\theta} x_0\| < \e$.
	Since $\|T x_0\| = \|e^{i\theta} x_0\| = 1$, by Lemma \ref{maria} there is a surjective linear isometry $R \in \mathcal{L}(H)$ which maps $T x_0$ to $e^{i\theta} x_0$ and $\|R - \Id_H\| < \e$. Let us notice the obvious fact that a rotation of $T$ is also unitary if $T$ is unitary. 
	Define $S := R \circ T \in \mathcal{L}(H)$. Then $S$ is unitary, $\nu(S) = \|S\| = 1$, $|\langle Sx_0, x_0 \rangle| = |\langle e^{i\theta} x_0, x_0 \rangle| = 1$, and $\|S - T\| = \| R \circ T - T\| \leq \|R - \Id_H\|\|T\| = \|R - \Id_H\| < \e$.

It remains to prove \ref{nu_normal_operators} and \ref{nu_cpt_normal_operators}. Let $\e \in (0,\frac{1}{2})$ be given. Suppose that $T \in \mathcal{L}(H)$ is a normal operator with $\|T\| = \nu(T) = 1$ and $|\langle Tx_0, x_0 \rangle| > 1-\e$ for some $x_0 \in S_H$. If $\theta \in \mathbb{R}$ is such that $\langle Tx_0, x_0 \rangle e^{i \theta} = |\langle Tx_0, x_0\rangle|$, then 
	\begin{eqnarray*}
	\|T(e^{i \theta} x_0) - x_0\|^2 &=& \langle T(e^{i\theta} x_0) - x_0, T(e^{i\theta} x_0) - x_0\rangle \\
&=& \|Tx_0\|^2 + \|x_0\|^2 - \langle T(e^{i\theta} x_0), x_0\rangle -\langle  x_0, T(e^{i \theta} x_0) \rangle \\
&<& 2 - 2(1-\e) = 2\e.
	\end{eqnarray*}
That is, $\| T(e^{i\theta} x_0) - x_0\| < \sqrt{2\eps}$. 
Let $E$ be the corresponding spectral measure of $T$ and consider the following orthogonal decomposition: $x_0 = x_1 + x_2$, where 
\[
x_1 = E(\sigma(T)\setminus B(1-\sqrt{2\e}))(x_0), ~~ x_2 = E(\sigma(T)\cap B(1-\sqrt{2\e}))(x_0)
\]
and let $N_1$ and $N_2$ be defined as 
\[
N_1 =\int_{\sigma(T)\setminus B(1-\sqrt{2\eps})} \frac{z}{|z|}\,dE(z) \quad \text{and} \quad  N_2 = \int_{\sigma(T)\cap B(1-\sqrt{2\eps})} z\,dE(z), 
\]
where $B(r)$ denotes the closed disk centered at the origin with radius $r>0$ in $\mathbb{C}$.
From \cite[Theorem 3.1]{CD}, we notice that $\|x_1\| \geq 1-\sqrt{2\e}$, $\|x_2\| \leq \sqrt[4]{2\e}$ and moreover if we let $x_\eps = x_1 /\|x_1\|$, then $\|x_\eps - x_0\| \leq \sqrt{2\e} + \sqrt[4]{2\e}$. This implies that 
	\begin{eqnarray*}
	\|T(e^{i\theta} x_{\e}) - x_{\e}\| &=& \frac{1}{\|x_1\|} \|T(e^{i\theta} x_{1}) - x_{1}\| \\
	&\leq& \frac{1}{\|x_1\|}\left( \|T(e^{i\theta} x_{0}) - x_{0}\| + \|T(e^{i\theta} x_2) - x_2\| \right) \\
	&\leq& \frac{1}{1-\sqrt{2\e}} \left( \sqrt{2\e} + 2 \sqrt[4]{2\e}\right).
	\end{eqnarray*}
Note now that 
	\[
	\| N_1 x_{\e} \|^2 = \left\langle E(\sigma(T)\setminus B(1-\sqrt{2\e})) x_{\e}, x_{\e}\right\rangle 
	= \langle x_\e, x_\e \rangle = 1,
	\]
because $x_\e$ belongs to the range of $E(\sigma(T)\setminus B(1-\sqrt{2\e}))$. From \cite[Lemma 2.4]{CD}, we see that the range space $K := \ran E(\sigma(T)\setminus B(1-\sqrt{2\e}))$ is a closed subspace of $H$. By Lemma \ref{maria}, there is a surjective isometry 
$\widetilde{R} \in \mathcal{L}(K) $ such that $\widetilde{R} x_{\e} =  N_1( e^{i\theta}x_{\e})$
and $\|\widetilde{R} - \Id_{K}\| =  \left\| x_\e - N_1 (e^{i\theta} x_{\e})\right\|,$
because $\ran N_1 \subset K$. Since $E(\sigma(T)\setminus B(1-\sqrt{2\e}))$ is a self-adjoint projection, we can observe that $ H  = K \oplus K', $
where $K' :=  \ker (E(\sigma(T)\setminus B(1-\sqrt{2\e})))$.

Let us define the operator $R \in \mathcal{L}(H)$ as $R = \widetilde{R} \oplus \Id_{K'}$, that is, $R(x+y) = \tilde{R} (x) + y$ for $x \in K$ and $y \in K'$. Since $\widetilde{R}$ is a surjective isometry, so is $R$. The adjoint $R^*$ of $R$ is given by $R^* = (\widetilde{R})^* \oplus \Id_{K'}$. We claim that the operator $R^* \circ N_1$ is also a normal operator. To see this, note first that 
\begin{equation*}
	(R^*  N_1) (R^* N_1)^* = R^* E(\sigma(T)\setminus B(1-\sqrt{2\e})) R, 
\end{equation*}	
and   
\begin{equation*}
(R^* N_1)^* (R^* N_1) = E(\sigma(T)\setminus B(1-\sqrt{2\e})).
\end{equation*}
Now, if $x \in K$, we have 
	\begin{eqnarray*}
	[R^* E(\sigma(T)\setminus B(1-\sqrt{2\e})) R](x) = R^* (Rx) = x \quad \text{and} \quad E(\sigma(T)\setminus B(1-\sqrt{2\e})) (x) = x. 
	\end{eqnarray*}
	If $x \in K'$, we have 
	\begin{eqnarray*}
	[R^* E(\sigma(T)\setminus B(1-\sqrt{2\e})) R](x) = R^* (E(\sigma(T)\setminus B(1-\sqrt{2\e})) x)  = 0, \quad E(\sigma(T)\setminus B(1-\sqrt{2\e})) (x) = 0. 
	\end{eqnarray*} 
	This observation shows that $R^* E(\sigma(T)\setminus B(1-\sqrt{2\e})) R= E(\sigma(T)\setminus B(1-\sqrt{2\e}))$ and the claim is proved.

We define the operator $S \in \mathcal{L}(H)$ by 
	\[
	S = R^*\circ N_1 + N_2.
	\]
To see that $S$ is a normal operator, it suffices to check that $R^*\circ N_1$ and $N_2$ commute with each other. Indeed,	from
	\begin{eqnarray*}
	\ran N_2 \subset \ker N_1 \quad \text{and} \quad  
	\ran R^* N_1 \subset \ran  E(\sigma(T)\setminus B(1-\sqrt{2\e}))  \subset \ker N_2,
	\end{eqnarray*}
	we obtain that $(R^* N_1) N_2 =0 = N_2 (R^* N_1)$. 
	 Moreover, 
	\begin{eqnarray*}
	\|S x\|^2 &=& \left\| R^* N_1 x + N_2 x \right\|^2 \\
	&=& \left\| R^*N_1 x_1\right\|^2 + \left\| N_2 x_2  \right\|^2 \\
	&\leq& \|x_1\|^2 + \|x_2\|^2 = \|x\|^2 
	\end{eqnarray*}
	for $x = x_1 + x_2 \in K \oplus K'$, because $\ran  R^*N_1 \subset K$ and $\ran N_2 \subset K'$. This implies that $\|S\|\leq 1$. Now, note that 
	\[
	|\langle Sx_{\e}, x_{\e}\rangle| = \left| \left\langle R^* N_1 x_{\e}, x_{\e} \right\rangle \right| 
	= \left| \left\langle N_1 x_{\e}, R x_{\e} \right\rangle \right| = 1.  
	\]
	This shows that $\nu(S) \geq 1$; hence $\|S\|=\nu(S)=1$. To assert that $S$ is the desired normal operator, it only remains to show that $S$ is close to $T$. Indeed, 
	\begin{eqnarray*}
	\| S - T \| &=& \left\|  R^* N_1  - \int_{\sigma(T)\setminus B(1-\sqrt{2\e})} z \, dE(z) \right\| \\
	&\leq& \| \widetilde{R} - \Id_{K} \| + \left\|  \int_{\sigma(T)\setminus B(1-\sqrt{2\e})} \left(\frac{z}{|z|}  -  z\right) \, dE(z)\right\| \\
	&\leq& \| \widetilde{R} - \Id_{K} \| + \sqrt{2\e}, 
	\end{eqnarray*}
	because $|z/|z| -z| \leq \sqrt{2\e}$ for all $z \in \sigma(T)\setminus B(1-\sqrt{2\e})$. Since 
	\begin{eqnarray*}
	\| \widetilde{R} - \Id_{K} \| &=& \left\| x_\e - N_1 (e^{i \theta} x_{\e}) \right\| \\
	&\leq& \| x_{\e} - T(e^{i \theta} x_{\e}) \| + \left\| T(e^{i \theta} x_{\e} ) - N_1 (e^{i \theta} x_{\e}) \right\|\\
	&\leq& \frac{1}{1-\sqrt{2\e}} \left( \sqrt{2\e} + 2 \sqrt[4]{2\e}\right) + \left\| \left( \int_{\sigma(T)\setminus B(1-\sqrt{2\e})} \left(z - \frac{z}{|z|}\right) \, dE(z)\right) (e^{i\theta} x_{\e}) \right\| \\
	&\leq& \frac{1}{1-\sqrt{2\e}} \left( \sqrt{2\e} + 2 \sqrt[4]{2\e}\right) + \sqrt{2\e}, 
	\end{eqnarray*} 
	we conclude that 
	\begin{eqnarray*}
	\|S-T\| \leq \frac{1}{1-\sqrt{2\e}} \left( \sqrt{2\e} + 2 \sqrt[4]{2\e}\right) + 2\sqrt{2\e}.
	\end{eqnarray*} 
	In summary, we construct the normal operator $S$ and $x_{\e} \in S_H$ satisfying:
	\begin{equation*}
\nu(S) = |\langle S x_{\e}, x_{\e}\rangle| = 1,\ \ \ \ \|x_{\e} - x_0\| \leq \sqrt{2\e}+\sqrt[4]{2\e}, \ \ \ \mbox{and} \ \ \ \|S - T\| \leq \frac{1}{1-\sqrt{2\e}} \left( \sqrt{2\e} + 2 \sqrt[4]{2\e}\right) + 2\sqrt{2\e}.
	\end{equation*}
Therefore, \ref{nu_normal_operators} follows again by using Theorem \ref{maintheorem}.

To prove \ref{nu_cpt_normal_operators}, we only need to show that the operator $S$ in the proof of \ref{nu_normal_operators} is compact when $T$ is compact and normal. To prove that $S$ is compact, since $S = R^*\circ N_1 + N_2$, it suffices to show that $N_1$ and $N_2$ are compact. Recall that 
\begin{equation*} 
	N_1 = \int_{\sigma(T)\setminus B(1-\sqrt{2\eps})} \frac{z}{|z|} \, dE(z) \quad \mbox{and} \quad 
	N_2 = \int_{\sigma(T) \cap B(1-\sqrt{2\eps})} z\,dE(z)
\end{equation*}
and observe from Lemma \ref{lem2} that 
\begin{equation*}
\ran N_1 \subset \ran E(\sigma(T) \setminus B(1-\sqrt{2\eps})) 
\end{equation*}
is of finite dimension. Thus, $N_1$ is compact. To see that $N_2$ is compact, 
we let $0< \eps' < 1-\sqrt{2\eps}$ be given. Now note that 
	\begin{align*} 
		\int_{\sigma(T) \cap B(1-\sqrt{2\eps})} z \, dE(z) - \left(\int_{\sigma(T) \cap B(1-\sqrt{2\eps})} z \, dE(z) \right) E(\Delta_{\eps'}) &= \int_{\sigma(T)} z \, \chi_{B(1-\sqrt{2\eps})} (z) \, \chi_{B(\eps')} (z) \, dE(z) \\\\
		&=\int_{\sigma(T)} z \chi_{B(\eps')} (z) \, dE(z), \nonumber
	\end{align*}
	where $\Delta_{\epsilon'} = \{z\in\sigma(T):|z|>\eps'\}.$
It follows that 
	\begin{align*} 
		\left\| \int_{\sigma(T) \cap B(1-\sqrt{2\eps})} z \, dE(z) - \left(\int_{\sigma(T) \cap B(1-\sqrt{2\eps})} z \, dE(z) \right) E(\Delta_{\epsilon'}) \right\| &=\left\|\int_{\sigma(T)} z \chi_{B(\eps')} (z) \, dE(z) \right\| 
		\leq \eps'.
	\end{align*}
Since $0< \eps'< 1-\sqrt{2\eps}$ is arbitrary and $(\int_{\sigma(T) \cap B(1-\sqrt{2\eps})} z \, dE(z)) E(\Delta_{\eps'})$ is a finite rank operator, we conclude that $N_2$ is compact.
\end{proof}

As in Proposition \ref{referee1}, we would like to get more information from Theorem \ref{propsec2}. Notice first that the operator $\widetilde{S}$ which appears in the first part of the proof of Theorem \ref{propsec2} is obtained from a limit of a sequence of operators $\{T_n\}$ (see \eqref{Kn}). Moreover, the argument used in the proof of \cite[Proposition 4]{KLM1} allows us 
to choose such $\alpha_1, \dots, \alpha_n$ to be $1$ when we start with the assumption that $T$ is positive. 
Thus we have that 
	\begin{eqnarray*}
	\langle K_n x, x \rangle &=& \left\langle  \left( \frac{\eps}{4}\right) \langle x, x_1 \rangle x_1 + \cdots  \left( \frac{\eps}{4}\right)^n \langle x, x_n \rangle x_n, x \right\rangle  \\\\
	&=&  \left( \frac{\eps}{4}\right) |\langle x, x_1 \rangle |^2 + \cdots +  \left( \frac{\eps}{4}\right)^n |\langle x, x_n \rangle |^2 \geq 0 
	\end{eqnarray*} 
for every $x \in H$, so $T_n$ is a positive operator. It follows that $\widetilde{S}$ is a positive operator which satisfies  
	\begin{equation*}
		\langle \widetilde{S}x_{\infty}, x_{\infty} \rangle = 1, \ \ \ \|x_{\infty} - x_0\| < \eps, \ \ \  \mbox{and} \ \ \ \|\widetilde{S} - T\| < \eps.
	\end{equation*}
This also shows that the operator $T_{\infty}$ that appears in the proof of item \ref{nu_SvN_operators} of Theorem \ref{propsec2} is positive. On the other hand, we can argue as in \ref{nu_selfadj_operators} and \ref{nu_cpt_selfadj_operators} of Theorem \ref{propsec2} to get the last two items of the following result.

\begin{prop} \label{referee2} Let $H$ be a complex Hilbert space.
\begin{enumerate}[label=(\alph*)]
	\item $H$ has the BPBpp-$\nu$ for positive operators.
	\item $H$ has the BPBpp-$\nu$ for positive Schatten-von Neumann operators. 
	\item $H$ has the BPBpp-$\nu$ for compact positive operators.  
	\item $H$ has the BPBpp-$\nu$ for self-adjoint Schatten-von Neumann operators. 
	\end{enumerate}
\end{prop}

Comparing Proposition \ref{referee2} with Proposition \ref{referee1}, we see that it is missing the Bishop-Phelps-Bollob\'as point property for numerical radius for normal Schatten-von Neumann operators. Since this result requires a little more of effort, we highlight it in the next proposition followed by its proof. 

\begin{prop}
A complex Hilbert space $H$ has the BPBpp-$\nu$ for normal Schatten-von Neumann operators. 
 \end{prop}

\begin{proof}
Let $T$ be a normal Schatten-von Neumann operator with $\nu (T) = \|T\| = 1$ and $x_0 \in S_H$ be such that $|\langle Tx_0, x_0\rangle| > 1-\eps$. Suppose that $\sigma_p (T) \leq M$ for some positive number $M > 0$. Let $S = R^*\circ N_1 + N_2$, where $R, N_1$, and $N_2$ are the operators defined in the proof of \ref{nu_normal_operators} and \ref{nu_cpt_normal_operators} of Theorem \ref{propsec2}. Observe that 
\begin{align*}
\sigma_p (S - T) &= \sigma_p \left( R^*N_1 - \int_{\sigma(T)\setminus B(1-\sqrt{2\eps})} z \, dE(z) \right) \\
&\leq \sigma_p \left( R^* N_1 - N_1 \right) + \sigma_p \left( \int_{\sigma (T) \setminus B(1-\sqrt{2\eps})}  \left(\frac{1}{|z|} - 1 \right) z \, dE(z) \right) \\
&\leq \|\widetilde{R} - \Id_K \| \sigma_p (N_1) + \sigma_p \left( \int_{\sigma (T) \setminus B(1-\sqrt{2\eps})}  \left(\frac{1}{|z|} - 1 \right) z \, dE(z) \right).
\end{align*}
By definition of $N_1$, we have that 
\begin{eqnarray*}
\sigma_p (N_1) &=& \sigma_p \left( \int_{\sigma(T)\setminus B(1-\sqrt{2\eps})} \frac{z}{|z|} \, dE(z) \right) \\
&=& \sigma_p \left( \left(\int_{\sigma(T)} z \, dE(z) \right) \left( \int_{\sigma(T)} \frac{1}{|z|} \chi_{\Delta_{\eps}} \, dE(z) \right)  \right) \\
&\leq& \left\| \int_{\sigma(T)} \frac{1}{|z|} \chi_{\Delta_{\eps}} \, dE(z)  \right\| \sigma_p(T) \\
&\leq& \frac{M}{1-\sqrt{2\eps}},
\end{eqnarray*}
where $\Delta_{\eps} = \{z\in\sigma(T):|z|> 1-\sqrt{2\epsilon} \}.$ Similarly, we can see that 
\begin{eqnarray*}
 \sigma_p \left( \int_{\sigma (T) \setminus B(1-\sqrt{2\eps})}  \left(\frac{1}{|z|} - 1 \right) z \, dE(z) \right) &\leq& \left\| \int_{\sigma(T)} \left( \frac{1-|z|}{|z|} \right) \chi_{\Delta_{\eps}} \, dE(z) \right\| \sigma_p (T) \\
 &\leq& \frac{M\sqrt{2\eps}}{1-\sqrt{2\eps}}.
\end{eqnarray*} 
It follows, in particular, that $S$ is a normal Schatten-von Neumann operator and
\[
\sigma_p (S-T) \leq \left(\frac{1}{1-\sqrt{2\e}} \left( \sqrt{2\e} + 2 \sqrt[4]{2\e}\right) + \sqrt{2\e}\right) \frac{M}{1-\sqrt{2\eps}} + \frac{M\sqrt{2\eps}}{1-\sqrt{2\eps}}. 
\] 
\end{proof}



			

\proof[Acknowledgements]

The authors would like to thank Miguel Mart\'in for kindly answering some inquiries about this topic. Also, they wish to express their gratitude to the anonymous referees for the careful reading of the manuscript and for his/her suggestions.

\end{document}